\documentclass[abstract=false]{scrartcl}

\usepackage{preamble}

\title{The universal coCartesian fibration}
\author{Denis-Charles Cisinski and Hoang Kim Nguyen}
\date{}
\begin{document}
\maketitle
\begin{abstract}
\noindent\textbf{\textsf{Abstract.}}
We give a new proof of
the straightening/unstraightening correspondence
by proving a generalization of the univalence property
of the universal coCartesian fibration.
\end{abstract}
\tableofcontents


\section*{Introduction}
A fundamental tool of \(\infty\)-category theory
is that there is an \(\infty\)-category \(\Q\) of small \(\infty\)-categories
so that coCartesian fibrations with small fibers of the form $X\to A$ correspond
to functors \(A\to\Q\). And this correspondence is implemented through an equivalence of categories.
This is
what is coined as straightening/unstraightening by Lurie in his foundational work.
Whether it is in his published monograph \emph{Higher Topos Theory} \cite{lurie}
or in the \emph{Kerodon} \cite{kerodon},
the \(\infty\)-category \(\Q\) of small \(\infty\)-categories is constructed
in two steps: we construct a category enriched in Kan complexes of small \(\infty\)-categories
and apply the homotopy coherent nerve to the latter. From this point of view,
proving the straightening/unstraightening
correspondence will thus necessarily involve dealing with interpreting at least some constructions
of \(\infty\)-category theory through the homotopy coherent nerve or its left adjoint.
We propose here an alternative construction in the setting of quasi-categories which does not use
any kind of homotopy coherent nerve. It is an adaptation of simpler version of
the straightening/unstraightening
correspondence in the case where the fibers are \(\infty\)-groupoids already developed in \cite{cisinskibook}, which is itself an adaptation of the construction of a univalent universe of Kan complexes in the sense of homotopy type theory by Kapulkin, Lumsdaine and Voevodsky \cite{klvsimplicial}, and
further generalized by Shulman \cite{shulmanelegant,shulman}.
Similar constructions and proofs could also be performed directly in the language of complete Segal spaces,
completing the work of Rasekh \cite{rasekh}. Whereas the point of view on the
\(\infty\)-category of small \(\infty\)-categories as a homotopy coherent nerve has the advantage of giving rather quickly a concrete description on this fundamental object
(and this can be done rather efficiently, as explained for instance in \cite{short_StrUnStr}) this kind of construction is not fully satisfying: it is now clear that the theory of \(\infty\)-categories goes much beyond being some form of theory of categories enriched in homotopy types and there is no reason to impose that we have a set (nor a \(0\)-truncated object) of objects. Indeed, we already have several
examples of new contexts (new topoi) in which we want to do mathematics, including applying homotopical methods such as \(\infty\)-category-theoretic constructions, e.g. equivariant stable homotopy theory, global homotopy theory, or condensed mathematics,
and a systematic study of the semantic interpretation of \(\infty\)-category theory in any \(\infty\)-topos is already developed in the work of
Martini and Wolf \cite{martini1,martini2,MW1,MW2}.
There are already formal theories of \(\infty\)-category as well:
in the direction of cosmoi (from the point of view of categories enriched in quasi-categories), as in the work of Riehl and Verity~\cite{RV}
(with its own version of
straightening/unstraightening correspondence \cite{RVcocart}), or in the direction of
directed type theory, as in the work of North~\cite{north} and Riehl and Shulman~\cite{RS}. Our contribution here is to provide a formulation of
the property of \emph{directed univalence} which is a property which is
not obviously equivalent to the straightening/unstraightening
correspondence but seems to encode all the information needed to
deduce it, while it makes sense in any formal context in which
we can express the language of \(\infty\)-category theory. We hope this
can be useful in practice as well as to study the foundations of
directed type theory.

In this article, we discuss the universal coCartesian fibration
once again. We will mainly focus on its construction in the classical setting
of \(\infty\)-category theory which is nowadays the one of quasi-categories.
From this perspective, this may be seen
as a complement of the first name author's book~\cite{cisinskibook}, and
as an addition to the general literature giving an alternative access to such a fundamental construction, in the spirit of \cite{short_StrUnStr}, which we hope to be useful: we have tried our best to make these notes usable
by any reader eager to learn and/or use this theory in their own research.
But we also focus secretly on a reformulation
of the straightening/unstraightening correspondence
that could easily be formulated (or implemented) through a formal language
in any abstract directed type theory: directed univalence. The way we formulate
it is more general than the one considered by Licata and Weaver~\cite{LW}: if
we take apart the constructive aspects of their work, what they do is
more in the spirit of \cite[§5.2--5.3]{cisinskibook}, since they focus on
left fibrations (coCartesian fibration with fibers in \(\infty\)-groupoids).
In the context of \(\infty\)-category theory,
we prove that the universal
coCartesian fibration defined tautologically (Definition \ref{def:univcoCart})
is directed univalent (Theorem \ref{diruniv}). More generally,
we define what it means for a coCartesian fibration
to be \emph{directed univalent}
(Definition \ref{def:directed univalent}) and
prove that a coCartesian fibration is directed
univalent if and only if its classifying functor is a full embedding of its
codomain into the \(\infty\)-category theory of \(\infty\)-categories
(Theorem \ref{thm:directedunivfull}). Since $\infty$-groupoids form
a full subcategory of the \(\infty\)-category theory of \(\infty\)-categories,
this gives yet another proof that the universal left fibration is directed
univalent, for instance (thus recovering some results of \cite{cisinskibook}).
We also explain how to deduce the straightening/unstraightening correspondence
in the following form: given any small \(\infty\)-category \(A\), there is
a canonical functor \[
\gamma_A\colon \mSet/A^\sharp \to \Fun(A,\Q)
\]
from the category \(\mSet/A^\sharp\) of marked simplicial sets over \(A\)
that exhibits \(\Fun(A,\Q)\) with the \(\infty\)-category theoretic
localization of \(\mSet/A^\sharp\) at the covariant weak equivalences
(the class of weak equivalences of the covariant model structure, the
fibrant objects of which precisely are the coCartesian fibrations over
\(A\)).

In this paper, we will start almost from scratch: we mainly
need basic results on the properties of coCartesian fibrations
that are explained in Lurie's \emph{Higher Topos Theory}~\cite[§2.4]{lurie}.
Many of the arguments in this paper
are adaptations of arguments developed
in detail in \cite[Chapter~5]{cisinskibook} to the setting of marked
simplicial sets. In particular,
we will also use covariant model structures on marked simplicial sets
whose fibrant objects are the coCartesian fibrations.
Such model structures are described in \cite[§3.1]{lurie}
and are revisited by the
second named author in \cite{contrakim}. Technical but
fundamental properties of these model structures (expressing
in secret exponentiability of coCartesian fibrations as well
as Beck-Chevalley properties of Kan extensions along coCartesian
fibrations) are furthermore developed in a companion paper by
the second named author~\cite{cocartesiankim}; such results will
be recalled explicitly when needed.
Only in the last section, in order to deduce
the straightening/unstraightening correspondence (Theorem~\ref{thm:straightening}),
will we use less elementary results about localizations of
model structures from the last chapter of \cite{cisinskibook}.
We also explain how to interpret the straightening/unstraightening correspondence
in a purely internal way (Theorem \ref{thm:straightening2}).

\emph{Acknowledgments}. We benefited from valuable discussions with Emily Riehl
and Nima Rasekh
(who insisted that we should highlight Theorem \ref{thm:directedunivfull}
if we want to convince any one that we indeed speak of some kind of univalence). The second named author greatly benefited from discussions with Benjamin D\"unzinger and Johannes Glossner concerning the universal morphism classifier.
Most of the results of this paper were announced in talks a couple of years
ago, in the \emph{HoTT Electronic Seminar}
and in the workshop \emph{GeoCat 2020} (an IJCAR-FSCD-2020 satellite workshop),
which were great opportunities
to discuss our results. This article was written while both authors were members
of the SFB~1085 ``Higher Invariants'' funded by the Deutsche Forschungsgemeinschaft (DFG).


\section{Reminder on coCartesian fibrations}

\begin{definition}\label{cartedge}
Let $p\colon X \to A$ be an inner fibration of simplicial sets and let $f\colon x\to y$ be an edge in \(X\). Then \(f\) is called \(p\)-coCartesian if for all $n \geq 2$ and all lifting problems of the form
	\[
	\begin{tikzcd}
	\Delta^{\{0,1\}}\drar{f}\dar & \\
	\Lambda^n_n \rar \dar & X\dar{p}\\
	\Delta^n \rar \urar[dashed] & A.
	\end{tikzcd}
	\]
there exists a lift as indicated. The map \(p\) is a coCartesian fibration if for all lifting problems of the form
\[
\begin{tikzcd}
\Delta^{\{0\}} \rar \dar & X\dar{p}\\
\Delta^1 \rar \urar[dashed] & A,
\end{tikzcd}
\]
there exists a lift as indicated, which is $p$-coCartesian.
\end{definition}

\begin{lemma}\label{cocartovereq}
Let \(p\colon X\to A\) be an inner fibration between \(\infty\)-categories. Let \(f\colon \Delta^1 \to X\) be a morphism. Then the following are equivalent:
\begin{itemize}
\item \(f\) is an equivalence,
\item \(f\) is coCartesian and \(p(f)\) is an equivalence.
\end{itemize}
\end{lemma}

\begin{lemma}\label{twothreecocartesian}
Let \(X\to A\) be an inner fibration between \(\infty\)-categories. Let
\[
\begin{tikzcd}
{} & \cdot \drar{f} &\\
\cdot \urar{g} \ar{rr}{h} & & \cdot
\end{tikzcd}
\]
be a commutative triangle in \(X\). Suppose \(g\) is coCartesian, then \(f\) is Cartesian if and only if \(h\) is Cartesian.
\end{lemma}

\begin{proof}
This is \cite[Proposition 2.4.1.7]{lurie}.
\end{proof}

\begin{lemma}\label{localtoglobal}
Let \(X\to A\) be a coCartesian fibration. Then an edge is Cartesian if and only if it is locally Cartesian.
\end{lemma}

\begin{proof}
See \cite[Corollary 5.2.2.4]{lurie}.
\end{proof}



\begin{sub}
We next recall the coCartesian model structure. This has been first constructued in \cite{lurie}. We follow the slightly more general treatment in \cite{contrakim}. We denote by $\mathbf{sSet}^+$ the category of marked simplicial sets. Its objects are given by pairs $(A,E_A)$ where $A$ is a simplicial set and $E_A$ is a subset of its 1-simplices containing all the degenerate 1-simplices. We will in general denote a marked simplicial set by \(A^+\). The forgetful functor
\[
U\colon \mathbf{sSet}^+\to \mathbf{sSet}
\]
has both a left adjoint, denoted by $(A)^\flat$, and a right adjoint, denoted by $(A)^\sharp$. The simplicial set $(A)^\flat$ has precisely the non-degenerate edges marked, while $(A)^\sharp$ has all 1-simplices marked.
\end{sub}

\begin{definition}\label{markedgenerators}
We define the class of \emph{marked  left anodyne extensions} to be the smallest saturated class containing the morphisms
\begin{itemize}
	\item[(A1)] $(\Lambda^n_k)^\flat \to (\Delta^n)^\flat$ for $n \geq 2$ and $0<k<n$,
	\item[(A2)] $J^\flat \to J^\sharp$, 
	\item[(B1)] $\mcyl \times (\Delta^1)^\flat \cup \{0\}\times \mcyl \to \mcyl \times \mcyl$,
	\item[(B2)] $\mcyl \times (\partial \Delta^n)^\flat\cup \{0\} \times (\Delta^n)^\flat \to \mcyl \times (\Delta^n)^\flat$.
\end{itemize}
Here, the simplicial set $J$ is the nerve of the free-walking isomorphism.
\end{definition}

\begin{remark}
The class we have described here differs slightly from Lurie's definition of \emph{marked anodyne morphisms}. Nevertheless, they define the same saturated class as his marked anodyne morphisms.
\end{remark}

\begin{definition}
A \emph{marked left fibration} is a map of marked simplicial sets having the right lifting property with respect to the class of marked left anodyne extensions.
\end{definition}

\begin{remark}
Note that a map $(X,E_X)\to A^\sharp$ is a marked left fibration if and only if the underlying map of simplicial sets is a coCartesian fibration and the set $E_X$ is precisely the set of coCartesian edges.
\end{remark}

\begin{remark}\label{dualgenerators}
There is the 'dual' notion of marked \emph{right} anodyne extensions in which the sets (B1) and (B2) in Definition \ref{markedgenerators} are replaced by the sets
\begin{itemize}
	\item[(B1')] $\mcyl \times (\Delta^1)^\flat \cup \{1\}\times \mcyl \to \mcyl \times \mcyl$,
	\item[(B2')] $\mcyl \times (\partial \Delta^n)^\flat\cup \{1\} \times (\Delta^n)^\flat \to \mcyl \times (\Delta^n)^\flat$.
\end{itemize}
Accordingly we have the notion of marked \emph{right} fibration.
\end{remark}

\begin{theorem}\label{cocartmodelstructure}
Let $A^+$ be a marked simplicial set. Then there exists a left proper combinatorial model structure on $\mathbf{sSet}^+/A^+$ with the following description.
\begin{itemize}
	\item The cofibrations are the morphisms of marked simplicial sets, whose underlying map of simplicial sets is a monomorphism.
	\item The fibrant objects are the marked left fibrations with target $A^+$.
	\item The fibrations between fibrant objects are precisely the marked left fibrations.
\end{itemize}
\end{theorem}

\begin{proof}
This is \cite[Theorem 4.29]{contrakim}.
\end{proof}

\begin{sub}
Let \(A\) be a simplicial set. We will denote \(\coCart(A)\) the homotopy category of the coCartesian model structure on \(\mSet/A^\sharp\). For the next Lemma recall that a \emph{cellular marked right anodyne extension} is a map in the saturated class generated by the classes (B1') and (B2') of Remark \ref{dualgenerators}.
\end{sub}

\begin{lemma}\label{smooth}
Let \(K^+\to L^+\) be a cellular marked right anodyne extension and \(X^+\to L^+\) be a marked left fibration. Then the pullback
\[
K^+\times_{L^+} X^+\to X^+
\]
is marked right anodyne.
\end{lemma}

\begin{proof}
This is \cite[Theorem 4.45]{contrakim}.
\end{proof}

\begin{theorem}\label{cocartesianinvariance}
Let \(i\colon K\to L\) be a Joyal equivalence. Then the functor
\[
i_! \colon \mSet/K^\sharp \to \mSet/L^\sharp
\]
induces a Quillen equivalence of coCartesian model structures.
\end{theorem}

\begin{proof}
This is \cite[Corollary 3.4]{cocartesiankim}.
\end{proof}


\section{Extension properties}

\begin{sub}
This section is the technical heart of the article. We prove various extension properties for coCartesian fibrations. Throughout this section we fix a regular cardinal \(\kappa\) and we call a map of simplicial sets \(X\to A\) \(\kappa\)-small if the pullback along any map \(\Delta^n\to A\) is \(\kappa\)-small. We will call a map of marked simplicial sets \(\kappa\)-small if its underlying map of simplicial sets is \(\kappa\)-small.
\end{sub}

\begin{sub}
The first extension property concerns extending an equivalence of coCartesian fibrations along a monomorphism. The proof is adapted almost word for word from \cite{shulmanelegant} and \cite{klvsimplicial}. This extension property will imply univalence of the universal coCartesian fibration. An easy consequence of the extension property is that we can extend coCartesian fibrations along trivial cofibrations of the Joyal model structure which will imply fibrancy of the universe of \(\kappa\)-small coCartesian fibrations.
\end{sub}

\begin{lemma}\label{sizelemma}
Let \(i\colon A^+\to B^+\) be a monomorphism of marked simplicial sets. Then
\begin{enumerate}
	\item \(i_*\colon \mSet/A^+\to \mSet/B^+\) preserves trivial fibrations,
	\item the counit \(i^*i_*\Rightarrow id\) is an isomorphism, 
	\item if \(p\colon X^+\to A^+\) is \(k\)-small, so is \(i_*p\).
\end{enumerate}
\end{lemma}

\begin{proof}
Exactly the same proof as in \cite[Lemma 2.2.4]{klvsimplicial}. Note that the first item holds without the assumption that \(i\) is a monomorphism.
\end{proof}

\begin{theorem}\label{extension}
Suppose we have a monomorphism of simplicial sets \(i\colon K\to L\) and a pullback square
\[
\begin{tikzcd}
Y_0^\natural\rar \dar & Y_1^\natural\dar\\
K^\sharp\rar{i} & L^\sharp
\end{tikzcd}
\]
in which the vertical arrows are coCartesian fibrations. Suppose furthermore we have an equivalence of coCartesian fibrations
\[
\begin{tikzcd}
X_0^\natural\ar{rr}{w_0}\drar & & Y_0^\natural\dlar\\
& K^\sharp &
\end{tikzcd}
\]
Then we can complete the dotted arrows in the diagram
\[
\begin{tikzcd}[column sep=small]
X_0^\natural\drar{w_0}\ar[dashed]{rr}\ar[bend right]{ddr} & & X_1^\natural\drar[dashed]{w_1}\ar[dashed, bend right]{ddr} & \\
 & Y_0^\natural\dar \ar{rr} & & Y_1^\natural\dar\\
 & K^\sharp \ar{rr}{i} & & L^\sharp
\end{tikzcd}
\]
such that the square
\[
\begin{tikzcd}
X_0^\natural\rar \dar & X_1^\natural\dar\\
K^\sharp\rar{i} & L^\sharp
\end{tikzcd}
\]
is a pullback, the map \(X_1^\natural\to L^\sharp\) is a coCartesian fibration and \(w_1\) is an equivalence of coCartesian fibrations. Finally, if all coCartesian fibrations are \(\kappa\)-small, then so is \(X^\natural_1\to L^\sharp\).
\end{theorem}

\begin{proof}
We define \(X_1^+\xrightarrow{w_1}  Y^\natural_1\) as the pullback
\[
\begin{tikzcd}
X_1^+\rar \dar[swap]{w_1} & i_\ast X_0^\natural\dar\\
Y_1^\natural \rar & i_\ast Y_0^\natural
\end{tikzcd}
\]
The bottom horizontal map is given by the unit map \(Y^\natural_1\to i_\ast i^\ast Y_1^\natural\). Applying \(i^\ast\) and using the fact that \(i_\ast\) is fully faithful we see that \(w_1\) pulls back to \(w_0\). It remains to show that \(w_1\) is a coCartesian equivalence and that \(X_1^+\to L^\sharp\) is a coCartesian fibration.

Since \(w_0\) is a coCartesian equivalence between coCartesian fibrations, it factorizes as a marked left anodyne map followed by a trivial fibration by Theorem \ref{cocartmodelstructure}. Thus it suffices to show the assertion if \(w_0\) is in either one of these classes of maps. Let us first assume that \(w_0\) is a trivial fibration. Since \(i_\ast\) preserves trivial fibrations, by Lemma \ref{sizelemma} \(w_1\) is the pullback of a trivial fibration hence the assertion follows.

Now suppose that \(w_0\) is marked left anodyne. We claim that \(w_1\) is a left deformation retract, hence also marked left anodyne. First observe that since \(w_0\) is a marked left anodyne extension between fibrant objects, it is in particular a left deformation retract. Thus, there is a map \(r\colon Y^\natural_0\to X_0^\natural\) and a map
\[
h\colon \mcyl \times Y_0^\natural \to Y_0^\natural
\]
over \(K^\sharp\) such that
\begin{itemize}
	\item \(rw_0=id\)
	\item \(h_0=w_0r\) and \(h_1= id\)
	\item \(h\) is constant on \(X^\natural_0\)
\end{itemize}

Consider the pushout
\[
\begin{tikzcd}
X^\natural_0 \rar \dar[swap]{w_0} & X_1^+\dar\\
Y^\natural_0 \rar & P^+
\end{tikzcd}
\]
which induces a monomorphism \(j\colon P^+\to Y^\natural_1\). Now we consider the lifting problem
\[
\begin{tikzcd}
\{0\} \times Y_1^\natural \cup \mcyl \times P^+ \rar \dar & Y_1^\natural \dar\\
\mcyl \times Y_1^\natural \rar \urar[dashed] & L^\sharp
\end{tikzcd}
\]
The upper horizontal map is given by the identity on \(\{0\}\times Y_1^\natural\) and on \(\mcyl \times P^+\) it is the pushout of \(h\) on \(Y^\natural_0\) and the constant homotopy on \(\mcyl \times X_1^\natural\). The lower horizontal map is the constant homotopy. Since the map \(P^+\to Y^\natural_1\) is a monomorphism, the left vertical map is marked left anodyne hence a solution exists. The solution now exhibits \(X_1^+\) as a left deformation retract of \(Y_1^\natural\). It follows that \(w_1\) is in particular marked left anodyne, hence a coCartesian equivalence. Moreover the map \(X_1^+\to L^\sharp\) is a retract of the coCartesian fibration \(Y^\natural_1\to L^\sharp\), hence also a coCartesian fibration.

Finally, \(X_1^+\to L^\sharp\) is a pullback of \(\kappa\)-small coCartesian fibrations by virtue of Lemma \ref{sizelemma}, hence is \(\kappa\)-small. 
\end{proof}

\begin{corollary}\label{eqextprop}
Let \(p\colon X\to K\) be a coCartesian fibration and let \(i\colon K\to L\) be a trivial cofibration of the Joyal model structure, with \(L\) \(\kappa\)-small.
Then we can complete the dotted arrows in the square
\[
\begin{tikzcd}
X\rar[dashed] \dar[swap]{p} & Y\dar[dashed]{q}\\
K \rar{i} & L
\end{tikzcd}
\]
such that the square is a pullback and such that \(q\colon Y\to L\) is a coCartesian fibration. Moreover,
it is possible to perform this construction so that,
if \(p\) is \(\kappa\)-small, so is \(q\)\footnote{The assumption
on the size of \(L\) can be dropped. See Remark \ref{remark:dropsize} below.}.
\end{corollary}

\begin{proof}
We may choose a commutative square in marked simplicial sets
\[
\begin{tikzcd}
X^\natural\rar{j} \dar[swap]{p} & (Y')^\natural\dar{q'}\\
K^\sharp \rar{i} & L^\sharp
\end{tikzcd}
\]
in which \(j\) is marked left anodyne and \(q'\) is a coCartesian fibration. One checks that this can be done such that \(q'\) is \(\kappa\)-small if \(p\) is \(\kappa\)-small,
using the small object argument. By \cite[Corollary 3.4]{cocartesiankim} the unit map \(X^\natural \to i^*(Y')^\natural\) is a coCartesian equivalence over \(K\). We may now apply Theorem \ref{extension} to obtain a pullback square
\[
\begin{tikzcd}
X^\natural\rar{j} \dar[swap]{p} & Y^\natural\dar{q}\\
K^\sharp \rar{i} & L^\sharp
\end{tikzcd}
\]
which induces the desired pullback square in simplicial sets.
\end{proof}

\begin{sub}
The next extension property concerns extending morphisms of coCartesian fibrations along trivial cofibrations of the Joyal model structure. We first fix some notation. Let \(p\colon X\to A\) be a map of simplicial sets. We denote the pullback of \(p\) along a map \(K\to A\) by
\[
\begin{tikzcd}
X_K\rar \dar & X\dar{p}\\
K\rar & A
\end{tikzcd}
\]
\end{sub}

\begin{proposition}\label{innerhornmorphextension}
Let \(K\to L\) be a trivial cofibration of the Joyal model structure. Suppose we have a pair of coCartesian fibrations 
\[
\begin{tikzcd}
X\drar & & Y\dlar\\
& L & 
\end{tikzcd}
\]
Suppose furthermore that there is a map
\[
\begin{tikzcd}
X_{K}\drar\ar{rr} & & Y_{K}\dlar\\
& K & 
\end{tikzcd}
\]
preserving coCartesian edges. Then there is a map
\[
\begin{tikzcd}
X\drar\ar{rr} & & Y\dlar\\
& L & 
\end{tikzcd}
\]
preserving coCartesian edges extending the map over \(K\).
\end{proposition}

\begin{proof}
We need to construct the following extension in marked simplicial sets
\[
\begin{tikzcd}[row sep=small,column sep=small]
X_{K}^\natural\drar \ar{rr} \ar{ddr} & & Y_{K}^\natural\drar\ar{ddl} & \\
{} & X^\natural\ar[dashed]{rr}\ar{ddr} & & Y^\natural\ar{ddl}\\
{} & K^\sharp\drar & & \\
{} & {} & L^\sharp & 
\end{tikzcd}
\]
By Theorem \ref{cocartesianinvariance} and \cite[Lemma 2.31]{contrakim}, the map \(X^\natural_K\to X^\natural\) is marked left anodyne, hence the desired extension exists.
\end{proof}

\begin{sub}
Under further assumptions, we can also extend morphisms of coCartesian fibrations along outer horn inclusions.
\end{sub}

\begin{theorem}\label{lefthornmorphextension}
Suppose we have a pair of coCartesian fibrations 
\[
\begin{tikzcd}
X\drar & & Y\dlar\\
& \nsimplex & 
\end{tikzcd}
\]
and assume that the pullbacks
\[
\begin{tikzcd}
X_{\Delta^{\{0,1\}}}\drar & & Y_{\Delta^{\{0,1\}}}\dlar\\
& \Delta^{\{0,1\}} & 
\end{tikzcd}
\]
have the property that an edge is coCartesian if and only if it is Cartesian. Suppose furthermore that there is a map
\[
\begin{tikzcd}
X_{\Lambda^n_0}\drar\ar{rr} & & Y_{\Lambda^n_0}\dlar\\
& \Lambda^n_0 & 
\end{tikzcd}
\]
preserving coCartesian edges. Then there is a map
\[
\begin{tikzcd}
X\drar\ar{rr} & & Y\dlar\\
& \nsimplex & 
\end{tikzcd}
\]
preserving coCartesian edges, which extends the map over the outer horn.
\end{theorem}

\begin{proof}
We first construct the dotted arrow in the diagram
\[
\begin{tikzcd}[row sep=small,column sep=small]
X_{\Lambda^n_0}\drar \ar{rr} \ar{ddr} & & Y_{\Lambda^n_0}\drar\ar{ddl} & \\
{} & X\ar[dashed]{rr}\ar{ddr} & & Y\ar{ddl}\\
{} & \Lambda^n_0\drar & & \\
{} & {} & \nsimplex & 
\end{tikzcd}
\]
and then argue why it preserves coCartesian edges. We consider the following extension problem in marked simplicial sets:
\[
\begin{tikzcd}[row sep=small,column sep=small]
X_{\Lambda^n_0}^+\drar \ar{rr} \ar{ddr} & & Y_{\Lambda^n_0}^+\drar\ar{ddl} & \\
{} & X^+\ar[dashed]{rr}\ar{ddr} & & Y^+\ar{ddl}\\
{} & (\Lambda^n_0)^+\drar & & \\
{} & {} & (\nsimplex)^+ & 
\end{tikzcd}
\]
Here the edge \(\Delta^{\{0,1\}}\) in \((\Delta^n)^+\) and \((\Lambda^n_0)^+\) is marked and we mark the corresponding coCartesian edges in \(X^+\) and \(Y^+\). By assumption they are also the locally cartesian edges over \(\Delta^{\{0,1\}}\) and by Lemma \ref{localtoglobal} they are in fact cartesian. In particular, the vertical arrows are marked left as well as marked right fibrations. By Lemma \ref{smooth}, the map 
\[
X^+_{\Lambda^n_0}\to X^+
\]
is marked left anodyne hence the extension exists, inducing an extension of the original extension problem in simplicial sets.

The only thing left to show is that this extension problem maps coCartesian edges to coCartesian edges. The cases \(n\geq 3\) and \(n=1\) are clear. In case \(n=2\), the extension preserves coCartesian edges over \(\Delta^{\{0,1\}}\) and \(\Delta^{\{0,2\}}\) by construction. Thus we need to show that the extension maps coCartesian edges in \(X^+\) over \(\Delta^{\{1,2\}}\) to coCartesian edges in \(Y^+\). Thus consider such a coCartesian edge in \(X^+\)
\[
\begin{tikzcd}[row sep=small,column sep=small]
\cdot \ar{r}{f}&\cdot
\end{tikzcd}
\]
We may choose a Cartesian edge over \(\Delta^{\{0,1\}}\) and take composition to obtain a commutative triangle
\[
\begin{tikzcd}
{} & \cdot \drar{f} & \\
\cdot \urar \ar{rr} & & \cdot
\end{tikzcd}
\]
By assumption the Cartesian edge is also coCartesian hence, since coCartesian edges compose, Lemma \ref{twothreecocartesian}, all edges in the triangle are coCartesian. The extension maps this triangle to a triangle in \(Y^+\) in which all edges but the image of \(f\) is coCartesian. By Lemma \ref{twothreecocartesian} this implies that the image of \(f\) has also to be coCartesian.
\end{proof}

\begin{sub}
This Theorem has a companion dual version which we state individually since the proof that the extension preserves coCartesian edges is slightly different.
\end{sub}

\begin{theorem}\label{righthornmorphextension}
Suppose we have a pair of coCartesian fibrations 
\[
\begin{tikzcd}
X\drar & & Y\dlar\\
& \nsimplex & 
\end{tikzcd}
\]
and assume that the pullbacks
\[
\begin{tikzcd}
X_{\Delta^{\{n-1,n\}}}\drar & & Y_{\Delta^{\{n-1,n\}}}\dlar\\
& \Delta^{\{n-1,n\}} & 
\end{tikzcd}
\]
have the property that an edge is coCartesian if and only if it is Cartesian. Suppose furthermore that there is a map
\[
\begin{tikzcd}
X_{\Lambda^n_n}\drar\ar{rr} & & Y_{\Lambda^n_n}\dlar\\
& \Lambda^n_n & 
\end{tikzcd}
\]
preserving coCartesian edges. Then there is a map
\[
\begin{tikzcd}
X\drar\ar{rr} & & Y\dlar\\
& \nsimplex & 
\end{tikzcd}
\]
preserving coCartesian edges, which extends the map over the outer horn.
\end{theorem}

\begin{proof}
We need to solve the extension problem
\[
\begin{tikzcd}[row sep=small,column sep=small]
X_{\Lambda^n_n}^+\drar \ar{rr} \ar{ddr} & & Y_{\Lambda^n_n}^+\drar\ar{ddl} & \\
{} & X^+\ar[dashed]{rr}\ar{ddr} & & Y^+\ar{ddl}\\
{} & (\Lambda^n_n)^+\drar & & \\
{} & {} & (\nsimplex)^+ & 
\end{tikzcd}
\]
Where we mark the edge \(\Delta^{\{n-1,n\}}\) in \(\Delta^n\) and \(\Lambda^n_n\) and the corresponding coCartesian edges in \(X\) and \(Y\). This time, by assumption and Lemma \ref{localtoglobal} the vertical maps are marked left and marked right fibrations. By Lemma \ref{smooth} the map
\[
X^+_{\Lambda^n_n}\to X^+
\]
is marked right anodyne, hence an extension exists.

It remains to show that the solution respects coCartesian edges. As before, it suffices to show the case \(n=2\) and that the extension respects coCartesian edges over \(\Delta^{\{0,1\}}\). Let 
\[
\begin{tikzcd}[row sep=small,column sep=small]
\cdot \ar{r}{f}&\cdot
\end{tikzcd}
\]
be an edge in \(Y\) which is in the image of a coCartesian edge over \(\Delta^{\{0,1\}}\) in \(X\). We claim that it is equivalent to a coCartesian edge and hence itself coCartesian. To this end let \(\cdot \xrightarrow{g}\cdot\) be a coCartesian edge of \(X\) mapping to \(f\). Choosing coCartesian lifts and taking composition, we obtain a commutative triangle
\[
\begin{tikzcd}[row sep=small,column sep=small]
{} & \cdot \drar & \\
\cdot \urar{g} \ar{rr} & & \cdot
\end{tikzcd}
\]
in which each edge is coCartesian in \(X\) and this maps to a triangle
\[
\begin{tikzcd}[row sep=small,column sep=small]
{} & \cdot \drar & \\
\cdot \urar{f} \ar{rr} & & \cdot
\end{tikzcd}
\]
in which each edge except possibly \(f\) is coCartesian. Choosing a coCartesian edge over \(\Delta^{\{0,1\}}\), we obtain a factorization of the edge \(f\) into a coCartesian edge followed by an edge in the fiber over \(Y_1\). This leads to the following 4-simplex in \(Y\)
\[
\begin{tikzcd}[row sep=small,column sep=small]
{} & \cdot \dar \ar[bend left]{ddr}& \\
{} & \cdot \drar{\ast} & \\
\cdot \ar[bend left]{uur}{\ast}\urar{f} \ar{rr}{\ast} & & \cdot
\end{tikzcd}
\]
in which the edges marked with \((\ast)\) are known to be coCartesian. It follows from Lemma \ref{twothreecocartesian} that the outer edge is also coCartesian and by assumption it is also Cartesian. Thus again by Lemma \ref{twothreecocartesian} the vertical edge is Cartesian. Since it lies in the fiber it is actually an equivalence, proving that \(f\) is equivalent to a coCartesian edge.
\end{proof}


\section{The universal coCartesian fibration}

\begin{sub}
In this section we construct the universal coCartesian fibration, which classifies small coCartesian fibrations. The construction and proof that the codomain of the universal coCartesian fibration is an \(\infty\)-category first appeared in the second authors thesis \cite{kimthesis}. We give a new, more streamlined proof here.
\end{sub}

\begin{sub}
We continue to fix a regular cardinal \(\kappa\). We let \(\U(\kappa)\) be a universe classifying \(\kappa\)-small maps of simplicial sets and having the property that for any monomorphism \(K\to L\) and any diagram
\[
\begin{tikzcd}[row sep=tiny,column sep=small]
X\ar{dd}\drar \ar{rr} & & \U(\kappa)_\bullet\ar{dd}\\
 & Y \ar{dd} \urar[dashed] & \\
K\ar{rr} \drar & & \U(\kappa)\\
 & L\urar[dashed] & 
\end{tikzcd}
\]
in which the vertical arrows are \(\kappa\)-small and the squares are pullbacks, the dotted arrows exists such that the induced square is also a pullback. A construction of \(\U(\kappa)\) can be found in \cite{klvsimplicial} or \cite[Definition 5.2.3]{cisinskibook}. An \(n\)-simplex of \(\U(\kappa)\) is roughly speaking given by a \(\kappa\)-small map \(X\to \Delta^n\) and some extra data such that \(\U(\kappa)\) is indeed a simplicial set.
\end{sub}

\begin{definition}\label{def:univcoCart}
We define \(\Q_\kappa\subset \U(\kappa)\) to be the subobject spanned by \(\kappa\)-small coCartesian fibrations \(X\to \Delta^n\)
and the map \(\quniv\colon \Q_{\kappa,\bullet} \to \Q_\kappa\) to be the pullback
\[
\begin{tikzcd}
\Q_{\kappa,\bullet}\dar[swap]{\quniv}\rar & \U(\kappa)\dar\\
\Q_\kappa\rar & \U(\kappa)
\end{tikzcd}
\]
\end{definition}

\begin{sub}
The map \(\quniv\) classifies \(\kappa\)-small coCartesian fibrations, i.e. for any coCartesian fibration \(p\colon X\to A\), such that \(\Delta^n\times_AX\) is a
\(\kappa\)-small simplicial set for any map \(\Delta^n\to A\),
there exists a classifying map \(F\colon A\to \Q_\kappa\) and a pullback square
\[
\begin{tikzcd}
X\rar\dar[swap]{p} & \Q_{\kappa,\bullet}\dar{\quniv}\\
A\rar{F} & \Q_{\kappa}
\end{tikzcd}
\]
Sometime, we will omit the reference to \(\kappa\) and simply speak of small
coCartesian fibrations. In this case, we will also write
\(\quniv\colon \Q_{\bullet} \to \Q\) instead of
\(\quniv\colon \Q_{\kappa,\bullet} \to \Q_\kappa\).
\end{sub}

\begin{proposition}
The map \(\quniv\colon \Q_{\kappa,\bullet} \to \Q_\kappa\) is a coCartesian fibration.
\end{proposition}

\begin{proof}
Follows immediately by definition and the fact that being a coCartesian fibration can be checked on representables.
\end{proof}

\begin{remark}\label{remark:dropsize}
In particular, the map \(\quniv\) is an isofibration
(i.e. a fibration of the Joyal model structure). This means that,
in Corollary \ref{eqextprop}, we can drop the assumption that the codomain of \(i\)
is \(\kappa\)-small and still see that we can produce a \(\kappa\)-small coCartesian
fibration \(q\) whenever \(p\) is \(\kappa\)-small: indeed this slightly more
general form of Corollary \ref{eqextprop} simply means that the
map \(\quniv\colon \Q_{\kappa,\bullet} \to \Q_\kappa\) is an isofibration,
which can be checked by restricting
to lifting problem against trivial cofibrations between countable simplicial sets.
\end{remark}

\begin{sub}
We refer to \(\quniv\) as the universal coCartesian fibration. The main Theorem of this section is that the simplicial set \(\Q\) is fibrant.
\end{sub}

\begin{theorem}\label{universefibrancy}
The simplicial set \(\Q\) is an \(\infty\)-category.
\end{theorem}

\begin{proof}
We want to solve the lifting property
\[
\begin{tikzcd}
\Lambda^n_k\rar \dar & \Q\\
\Delta^n \urar[dashed] &
\end{tikzcd}
\]
for \(n\geq 2\) and \(0<k<n\). This amounts to showing the following extension property: For any small coCartesian fibration \(p\colon X\to \Lambda^n_k\) we can complete the square
\[
\begin{tikzcd}
X\rar[dashed] \dar[swap] & Y\dar[dashed]{q}\\
\innerhorn \rar & \nsimplex
\end{tikzcd}
\]
such that \(q\) is a small coCartesian fibration and the square is a pullback. This is precisely the content of Corollary \ref{eqextprop}.
\end{proof}

\begin{sub}
Let \(\spaces\subset\Q\) be the subobject on left fibrations and consider its maximal \(\infty\)-groupoid \(k(\spaces)\subset \spaces\). Pulling back the universal coCartesian fibration we obtain the diagram of pullbacks
\[
\begin{tikzcd}
k(\spaces_\bullet)\rar\dar[swap]{k(\puniv)} & \spaces_\bullet\dar{\puniv}\rar & \Q_\bullet\dar{\quniv}\dar\\
k(\spaces) \rar & \spaces \rar & \Q
\end{tikzcd}
\]
The map \(\puniv\) is the universal left fibration and the first author shows that its codomain is the \(\infty\)-category of spaces \cite[Theorem 7.8.9]{cisinskibook}. The map \(k(\puniv)\) is the universal Kan fibration constructed in \cite{klvsimplicial}.
\end{sub}

\begin{sub}
We next characterize coCartesian fibrations whose classifying map factorizes over the maximal \(\infty\)-groupoid of \(\Q\).
\end{sub}

\begin{proposition}\label{groupoidclass}
Let \(p\colon X\to A\) be a small coCartesian fibration. Then its classifying map factorizes over the maximal \(\infty\)-groupoid \(k(\Q)\) if and only if \(p\) is also a Cartesian fibration
and any edge of $X$ is coCartesian if and only if it is cartesian.
\end{proposition}

\begin{proof}
Suppose the classifying map factorizes over \(k(\Q)\). Note that the property of being a coCartesian and cartesian fibration with coCartesian edges precisely the cartesian edges and vice versa is stable under pullback. Also note that a coCartesian fibration over an \(\infty\)-groupoid satisfies this property by Lemma \ref{cocartovereq}. The assertion then follows since \(p\) is a pullback of a coCartesian fibration over an \(\infty\)-groupoid.

For the converse we first assume that \(A=\Delta^1\). We consider the map of marked simplicial sets \(X^\natural \to \mcyl\). By assumption, this is a marked left as well as a marked right fibration. Let \(J\) be the nerve of the free walking isomorphism. There exists a square
\[
\begin{tikzcd}
X^\natural \rar{j} \dar[swap]{p} & Y^\natural\dar{q}\\
\mcyl \rar & J^\sharp
\end{tikzcd}
\]
in which \(j\) is marked left anodyne and \(q\) is a marked left fibration. Since \(J\) is a groupoid \(q\) is also a marked right fibration. We claim that the induced map 
\[
\begin{tikzcd}[column sep=small]
X^\natural \ar{rr}\drar & & Y^\natural_{\Delta^1}\dlar\\
& \mcyl &
\end{tikzcd}
\]
is a coCartesian equivalence. In this case it is enough to show that it is a fiberwise equivalence. But since the inclusion of \(0\) into \(\mcyl\) as well as \(J^\sharp\) is cellular marked left anodyne and \(p\) and \(q\) are marked right anodyne, it follows from Lemma \ref{smooth} that the fibers of \(p\) and \(q\) over \(\{0\}\) are equivalent to \(X^\natural\) and \(Y^\natural\) respectively in the coCartesian model structure over the point. Since \(j\) is also a coCartesian equivalence over the point, the fibers over \(\{0\}\) are equivalent. The argument for the fibers over \(\{1\}\) is analogous using the fact that the inclusion of \(\{1\}\) into \(\mcyl\) and \(J^\sharp\) is cellular marked right anodyne.

Now by Theorem \ref{extension} we find an actual pullback square
\[
\begin{tikzcd}
X^\natural \rar{j} \dar[swap]{p} & Z^\natural\dar{r}\\
\mcyl \rar & J^\sharp
\end{tikzcd}
\]
thus the classifying map of \(p\) factorizes over the \(\infty\)-groupoid \(J^\sharp\) and has hence image in \(k(\Q)\). 

For general \(A\) we need to show that each edge in \(A\) is sent to an equivalence in \(\Q\). Let \(f\colon \Delta^1\to A\) be an edge, then by the above argument we obtain a commutative square
\[
\begin{tikzcd}
\Delta^1\dar[swap]{f} \rar & J\dar\\
A\rar & \Q
\end{tikzcd}
\]
Thus the edge maps to an equivalence in \(\Q\).
\end{proof}


\section{The universal morphism classifier}

\begin{sub}
The category of marked simplicial sets is locally cartesian closed. Given a marked simplicial set $A^+$ and two maps $p\colon X^+\to A^+$ and $q\colon Y^+\to A^+$, we denote the internal hom by 
\[
\pi_{X,Y}\colon\underline{\Hom}^+_{A^+}(X^+,Y^+)\to A^+.
\]
A map 
\[
K^+\to \mathrm{Hom}_{A^+}^+(X^+,Y^+)
\]
corresponds to the datum of a map $f\colon K^+ \to A^+$ and a map
\[
f^\ast X^+\to f^\ast Y^+
\]
over $K^+$. We denote by \(\underline{\mathrm{Eq}}^+_{A^+}(X^+,Y^+)\subset \Hom^+_{A^+}(X^+,Y^+)\) the subobject spanned by the coCartesian equivalences.
\end{sub}

\begin{definition}
We denote by $\Hom^\sharp_{A^+}(X^+,Y^+)$ and \(\underline{\mathrm{Eq}}^\sharp_{A^+}(X^+,Y^+)\) the simplicial sets spanned by the marked edges of $\underline{\Hom}^+_{A^+}(X^+,Y^+)$ and \(\underline{\mathrm{Eq}}^+_{A^+}(X^+,Y^+)\) respectively.
\end{definition}

\begin{sub}
In other words, a map of simplicial sets $K\to \Hom^\sharp_{A^+}(X^+,Y^+)$ corresponds by adjunction to a map $K^\sharp \to\Hom^+_{A^+}(X^+,Y^+)$. This is equivalent to specifying a map $f\colon K^\sharp\to A^+$ and a map $f^\ast X^+ \to f^\ast Y^+$ of marked simplicial sets over $K^\sharp$. In case $X$ and $Y$ are coCartesian fibrations, this amounts to specifying a map $K\to A$ and a map $f^\ast X \to f^\ast Y$ over $K$ which preserves coCartesian edges. We have a factorization over \(\underline{\mathrm{Eq}}^\sharp_{A^+}(X^+,Y^+)\) if and only if the the map over \(K\) is a coCartesian equivalence.
\end{sub}

\begin{proposition}\label{internalhomfibration}
Let \(p\colon X^\natural \to A^\sharp\) and \(q\colon Y^\natural \to A^\sharp\) be coCartesian fibrations. Then the map
\[
\begin{tikzcd}
\underline{\Hom}^\sharp_{A^\sharp}(X^\natural ,Y^\natural)\dar\\
A
\end{tikzcd}
\]
is a Joyal fibration.
\end{proposition}

\begin{proof}
Let \(K\to L\) be a Joyal trivial cofibration. The lifting problem
\[
\begin{tikzcd}
K\rar \dar & \underline{\Hom}^\sharp_{A^\sharp}(X^\natural ,Y^\natural)\dar\\
L\rar \urar[dashed] & A
\end{tikzcd}
\]
corresponds to the extension problem
\[
\begin{tikzcd}[row sep=small,column sep=small]
X_{K}^\natural\drar \ar{rr} \ar{ddr} & & Y_{K}^\natural\drar\ar{ddl} & \\
{} & X^\natural_L\ar[dashed]{rr}\ar{ddr} & & Y^\natural_L\ar{ddl}\\
{} & K^\sharp\drar & & \\
{} & {} & L^\sharp & 
\end{tikzcd}
\]
which admits a solution by Proposition \ref{innerhornmorphextension}.
\end{proof}

\begin{sub}
Let \(p\colon E\to B\) be a coCartesian fibration. We obtain two coCartesian fibrations
\[
\begin{tikzcd}
E^0\drar & & E^1\dlar\\
& B\times B &
\end{tikzcd}
\]
by pulling back along the projections \(B\times B\to B\).
\end{sub}

\begin{definition}\label{def:Homclassify}
We denote
\begin{itemize}
\item \(\underline{\Hom}_{B\times B}(E^0,E^1) := \Hom^\sharp_{(B\times B)^\sharp}((E^0)^\natural,(E^1)^\natural )\)
\item \(\underline{\mathrm{Eq}}_{B\times B}(E^0,E^1)
:= \mathrm{Eq}^\sharp_{(B\times B)^\sharp}((E^0)^\natural,(E^1)^\natural )\)
\end{itemize}
\end{definition}

\begin{sub}
Recall that a map \(X\to A\times B \) is called a bifibration if it satisfies the following two conditions:
\begin{itemize}
\item The lifting problem
\[
\begin{tikzcd}
\Lambda^n_0\rar \dar & X\dar\\
\Delta^n\rar\urar[dashed] & A\times B
\end{tikzcd}
\]
admits a solution whenever the edge \(\Delta^{\{0,1\}}\) is mapped to a degenerate edge in \(A\).
\item The lifting problem
\[
\begin{tikzcd}
\Lambda^n_n\rar \dar & X\dar\\
\Delta^n\rar\urar[dashed] & A\times B
\end{tikzcd}
\]
admits a solution whenever the edge \(\Delta^{\{n-1,n\}}\) is mapped to a degenerate edge in \(B\).
\end{itemize}
\end{sub}

\begin{theorem}\label{morphclassbifib}
The map 
\[
\begin{tikzcd}
\underline{\Hom}_{B\times B}(E^0,E^1)\dar\\
E\times E 
\end{tikzcd}
\]
is a bifibration.
\end{theorem}

\begin{proof}
Consider first the lifting problem 
\[
\begin{tikzcd}
\Lambda^n_0\rar \dar[swap]{i} & \underline{\Hom}_{B\times B}(E^0,E^1)\dar\\
\Delta^n\rar{(f,g)}\urar[dashed] & B\times B
\end{tikzcd}
\]
and assume that the edge \(\Delta^{\{0,1\}}\) is degenerate in the first component of \(B\times B\). This means that when projecting down to the first component we have a commutative square
\[
\begin{tikzcd}
\Delta^{\{0,1\}}\rar \dar & \Delta^0\dar\\
\Delta^n \rar{f} & B
\end{tikzcd}
\]
Consider the coCartesian fibration \(f^*E\to \Delta^n\) given by pulling back \(p\) along \(f\). The commutativity of the square and Proposition \ref{groupoidclass} imply that the restriction
\[
\begin{tikzcd}
f^*E_{\Delta^{\{0,1\}}}\dar\\
\Delta^{\{0,1\}}
\end{tikzcd}
\]
is a coCartesian as well as a cartesian fibration with coCartesian edges precisely the cartesian edges and vice versa. The lifting problem now corresponds to the extension problem
\[
\begin{tikzcd}[row sep=small,column sep=small]
f^*E_{\Lambda^n_0}\drar \ar{rr} \ar{ddr} & & g^*E_{\Lambda^n_0}\drar\ar{ddl} & \\
{} & f^*E\ar[dashed]{rr}\ar{ddr} & & g^*E\ar{ddl}\\
{} & \Lambda^n_0\drar & & \\
{} & {} & \nsimplex & 
\end{tikzcd}
\]
where the extension needs to preserve coCartesian edges. The existence of such an extension follows directly from Theorem \ref{lefthornmorphextension}. The lifting problem
\[
\begin{tikzcd}
\Lambda^n_n\rar \dar & \underline{\Hom}_{B\times B}(E^0,E^1)\dar\\
\Delta^n\rar\urar[dashed] & B\times B
\end{tikzcd}
\]
is solved analogously using Theorem \ref{righthornmorphextension}.
\end{proof}

\begin{sub}
We apply Definition \ref{def:Homclassify}
to the universal coCartesian fibration
\(\quniv\colon \Q_\bullet \to \Q\) and obtain the isofibration
\[
\begin{tikzcd}
\morphclass\dar{(s,t)}\\
\Q\times \Q 
\end{tikzcd}
\]
We refer to the map \(\morphclass\to \Q\times \Q\) as the universal morphism classifier and the map \(\eqclass \to \Q\times \Q\) as the universal equivalence classifier. Indeed a map \(A\to \morphclass\) corresponds to specifying a map of coCartesian fibrations
\[
\begin{tikzcd}
X\ar{rr}\drar & & Y\dlar\\
& A &
\end{tikzcd}
\]
preserving coCartesian edges. Similarly, a map to \(\eqclass\) corresponds to specifying a coCartesian equivalence between coCartesian fibrations. 
\end{sub}

\begin{corollary}\label{bifibration}
The universal morphism classifier is a bifibration.
\end{corollary}


\section{The Grothendieck construction on homotopy categories}

\begin{sub}
The goal of this section is to define the Grothendieck construction on homotopy categories. To this end consider a map of marked simplicial sets
\[
\begin{tikzcd}
W^+\dar{q}\\
A^+\times B^+
\end{tikzcd}
\]
Suppose we have a map \(p\colon X^+\to A^+\), then we define the marked simplicial set \(\Map_{A^+}^{B^+}(X^+,W^+)\) as the pullback
\[
\begin{tikzcd}
\Map_{A^+}^{B^+}(X^+,W^+)\rar \dar & \underline \Hom^+(X^+,W^+)\dar{q_\ast}\\
B^+\rar & \underline{\Hom}^+(X^+,A^+\times B^+)
\end{tikzcd}
\]
Here the bottom map is induced by the adjoint of 
\[
X^+\times B^+\xrightarrow{(p,id)} A^+\times B^+
\]
If \(q\) is a marked left fibration, then 
\[
\begin{tikzcd}
\Map_{A^+}^{B^+}(X^+,W^+)\dar\\
B^+
\end{tikzcd}
\]
is also a marked left fibration by \cite[Proposition 3.1.2.3]{lurie}, see also \cite[Lemma 4.33]{contrakim} for alternative proof.
\end{sub}

\begin{sub}
When \(B^+ =\Delta^0\), then \(\Map_{A^+}^{\Delta^0}(X^+,W^+)\) is simply the marked simplicial set of commuting triangles
\[
\begin{tikzcd}
X^+\ar{rr}\drar[swap]{p} & & W^+\dlar{q}\\
& A^+ &
\end{tikzcd}
\]
We are mainly interested in the case where \(X^\natural \to A^\sharp\) and \(W^\natural \to A^\sharp\) are coCartesian fibrations. In this case we write
\[
\Map_A^+(X,W) := \Map_{A^\sharp}^{\Delta^0}(X^\natural, W^\natural)
\]
to keep notation simple. Furthermore we denote \(\Map_A^\sharp(X,W)\) the simplicial set spanned by the marked edges. Note that \(\Map_A^\sharp(X,W)\) is an \(\infty\)-groupoid.
\end{sub}

\begin{sub}
Let \(A\) be a simplicial set and 
\[
\begin{tikzcd}
X\drar[swap]{p} & & Y\dlar{q}\\
& A &
\end{tikzcd}
\]
be coCartesian fibrations classified by \(F\colon A\to \Q\) and \(G\colon A\to \Q\) respectively. We then have a pullback square
\[
\begin{tikzcd}
\Map_A^\sharp(X,Y)\rar \dar & \Fun(A,\morphclass)\dar\\
\Delta^0\rar{(F,G)} & \Fun(A,\Q)\times \Fun(A,\Q)
\end{tikzcd}
\]
We define \(\equivfib^\sharp_A(X,Y)\) to be the pullback
\[
\begin{tikzcd}
\equivfib_A^\sharp(X,Y)\rar \dar & \eqclass\dar\\
\Map_A^\sharp(X,Y) \rar & \morphclass
\end{tikzcd}
\]
\end{sub}

\begin{sub}\label{grothendieckquillen}
Let \(B\) be a simplicial set and \(A\) be the nerve of a small category. In \cite[Theorem 5.2]{cocartesiankim} it is shown that there is a Quillen equivalence
\[
\lambda : \Fun(A,\mSet/B^\sharp) \leftrightarrow \mSet/A^\sharp\times B^\sharp : \rho
\]
where the slice categories are endowed with the coCartesian model structure and the functor category is endowed with the projective model structure. The left adjoint \(\lambda\) is given by taking the homotopy colimit while the right adjoint takes a map \(W^+\to A^\sharp \times B^\sharp\) to the functor
\[
a\mapsto \Map_{A^\sharp}^{B^\sharp}(A_{a/}^\sharp, W^+)
\]
with map to \(B^\sharp\) as constructed above.
\end{sub}

\begin{sub}
Consider a coCartesian fibration \(W\to \Delta^1\times A\). We obtain a map
\[
\begin{tikzcd}
\Map_{(\Delta^1)^\sharp}^{A^\sharp}((\Delta^1_{0/})^\sharp, W^\natural)\drar \ar{rr} & & \Map_{(\Delta^1)^\sharp}^{A^\sharp}((\Delta^1_{1/})^\sharp, W^\natural)\cong W^\natural_1\dlar\\
& A^\sharp &
\end{tikzcd}
\]
Since the inclusion \(\Delta^{\{0\}}\to (\Delta^1)^\sharp\) is marked left anodyne, we have a trivial fibration
\begin{equation}\label{sectionrect}
\Map_{(\Delta^1)^\sharp}^{A^\sharp}((\Delta^1_{0/})^\sharp, W^\natural)\to W_0^\natural
\end{equation}
over \(A^\sharp\), see \cite[Proposition 5.7]{cocartesiankim}, and choosing a section defines a map
\[
\begin{tikzcd}
W_0^\natural \ar{rr}\drar & & W_1^\natural\dlar\\
& A^\sharp &
\end{tikzcd}
\]
classified by a map \(W\to\morphclass\).
Applying this construction to the coCartesian fibration
classified by the evaluation map 
\[
\Delta^1\times \Fun(\Delta^1,\Q)\to \Q
\]
yields the following lift (with \(W=\Fun(\Delta^1,\Q)\)).
\begin{equation}\label{comparison}
\begin{tikzcd}
{} & \morphclass\dar{(s,t)}\\
\Fun(\Delta^1,\Q)\rar{(\mathit{ev}_0,\mathit{ev}_1)} \urar[dashed] & \Q\times \Q
\end{tikzcd}
\end{equation}
We get a commutative diagram
\[
\begin{tikzcd}
h(\Delta^1,\Q) \rar \dar & \eqclass\dar\\
\Fun(\Delta^1,\Q)\rar \drar & \morphclass\dar\\
& \Q\times \Q
\end{tikzcd}
\]
We will see that univalence of the universal coCartesian fibration is equivalent to the upper horizontal map being a categorical equivalence and we will define directed univalence to be a categorical equivalence of the middle horizontal map.
\end{sub}

\begin{proposition}\label{homotopyindependence}
The homotopy class of the lift (\ref{comparison}) in the Joyal model structure over \(\Q\times\Q\) is independent of the choice of section in (\ref{sectionrect}.)
\end{proposition}

\begin{proof}
Any two sections are \(\mcyl\)-homotopic over \(A^\sharp\). Thus we get a commutative triangle
\begin{equation}\label{eqofmorphclass}
\begin{tikzcd}
W_0^\natural \times \mcyl \ar{rr}\drar & &  W_1^\natural\times \mcyl\dlar\\
& A^\sharp \times \mcyl &
\end{tikzcd}
\end{equation}
This is equivalent to a map
\[
A\times \Delta^1 \to \morphclass
\]
We claim that this is a natural equivalence of functors, i.e. for each object \(a\) of \(A\), the induced map 
\[
\Delta^1 \to \morphclass
\]
defines an equivalence in \(\morphclass\). Note that this is the case if and only if the induced map
\[
\Delta^1 \to \Q\times \Q
\]
defines an equivalence. Such a map corresponds to a pair of coCartesian fibrations over \(\Delta^1\) and by Proposition \ref{groupoidclass} this defines an equivalence if and only if every coCartesian edge is Cartesian and vice versa. Now taking fibers in (\ref{eqofmorphclass}) this is clearly the case.
\end{proof}

\begin{sub}
Let \(A\) be a simplicial set. Let \(X\to A\) be a coCartesian fibration classified by a functor \(F\colon A\to \Q\) and \(Y\to A\) be a coCartesian fibration classified by \(G\colon A\to \Q\). Taking the fiber at \((F,G)\) in the diagram
\[
\begin{tikzcd}
\Fun(A,h(\Delta^1,\Q)) \rar \dar & \Fun(A,\eqclass)\dar\\
\Fun(A,\Fun(\Delta^1,\Q))\rar \drar & \Fun(A,\morphclass)\dar\\
& \Fun(A,\Q)\times \Fun(A,\Q)
\end{tikzcd}
\]
yields the commutative square
\[
\begin{tikzcd}
	k(A,\Q)(F,G)\rar \dar & \equivfib^\sharp_A(X,Y)\dar\\
	\Fun(A,\Q)(F,G)\rar & \Map_A^\sharp(X,Y)
\end{tikzcd}
\]
\end{sub}

\begin{sub}
Recall that \(\coCart(A)\) denotes the homotopy category of \(\mSet/A^\sharp\) endowed with the coCartesian model structure. The following Proposition can be viewed as a Grothendieck construction on homotopy categories. Here, we have to be careful about size: if $\kappa$ denotes an inaccessible
cardinal defining \(\Q=\Q_\kappa\), we define \(\coCart_\kappa(A)\)
as the full subcategory of \(\coCart(A)\) spanned by those objects
isomorphic to coCartesian
fibrations \(X\to A\) with \(\kappa\)-small fibers.
In the case where \(A\) is \(\kappa\)-small, we can identify \(\coCart_\kappa(A)\)
with the localization of \(\mSet_\kappa/A^+\) by the weak equivalence of
the model structure of Theorem \ref{cocartmodelstructure},
where \(\mSet_\kappa/A^+\) denotes the full subcategory of \(\mSet/A^+\)
spanned by coCartesian
fibrations with \(\kappa\)-small fibers.
\end{sub}

\begin{proposition}
There is a canonical functor
\[
ho(\Fun(A,\Q))\to \coCart_\kappa(A)
\]
sending a functor \(A\to \Q\) to the coCartesian fibration that it classifies.
\end{proposition}

\begin{proof}
The Grothendieck construction on morphisms is given by taking \(\pi_0\) of the map
\[
\Fun(A,\Q)(F,G)\to \Map_A^\sharp(X,Y)
\]
We need to show that this defines a functor. A composition in \(ho(\Fun(A,\Q))\) is represented by a coCartesian fibration
\[
\begin{tikzcd}
	W\dar\\
	\Delta^2\times A
\end{tikzcd}
\]
and we need to show that the induced triangle
\[
\begin{tikzcd}
	W_0^\natural \rar \drar & W_1^\natural\dar\\
	& W_2^\natural
\end{tikzcd}
\]
commutes up to homotopy in \(\mSet/A^\sharp\). But this triangle is homotopic to the triangle
\[
\begin{tikzcd}
	\Map_{(\Delta^2)^\sharp}^{A^\sharp}((\Delta^2_{0/})^\sharp, W^\natural)\rar \drar & \Map_{(\Delta^2)^\sharp}^{A^\sharp}((\Delta^2_{1/})^\sharp, W^\natural)\dar\\
	& \Map_{(\Delta^2)^\sharp}^{A^\sharp}((\Delta^2_{2/})^\sharp, W^\natural)
\end{tikzcd}
\]
which strictly commutes in \(\mSet/A^\sharp\).
\end{proof}

\begin{remark}
We will see that univalence is equivalent to the Grothendieck construction inducing an equivalence of maximal groupoids, while directed univalence is equivalent to the Grothendieck construction being an equivalence of categories.
\end{remark}


\section{Univalence}

\begin{sub}
We show that the universal coCartesian fibration satisfies the univalence axiom. This generalizes previous results that the universal left fibration \cite{cisinskibook} and the universal Kan fibration \cite{klvsimplicial} satisfy the univalence axiom.
\end{sub}

\begin{proposition}\label{targetfibration}
The target map
\[
t:\eqclass \to \Q
\]
is a trivial fibration.
\end{proposition}

\begin{proof}
We need to solve lifting problems of the form
\[
\begin{tikzcd}
K\rar \dar[swap]{i} & \eqclass\dar{t}\\
L\rar\urar[dashed] & \Q
\end{tikzcd}
\]
which corresponds to completing the diagram
\[
\begin{tikzcd}[column sep=small]
X_0^\natural\drar{w_0}\ar[dashed]{rr}\ar[bend right]{ddr} & & X_1^\natural\drar[dashed]{w_1}\ar[dashed, bend right]{ddr} & \\
 & Y_0^\natural\dar \ar{rr} & & Y_1^\natural\dar\\
 & K^\sharp \ar{rr}{i} & & L^\sharp
\end{tikzcd}
\]
such that the back square is a pullback, \(w_1\) is a coCartesian equivalence and \(X_1^\natural \to L^\sharp\) is a coCartesian fibration. This is precisely the content of Theorem \ref{extension}.
\end{proof}

\begin{theorem}[Univalence]
The diagonal 
\[
\Q\to \eqclass
\]
is a trivial cofibration of the Joyal model structure.
\end{theorem}

\begin{proof}
We have the commutative diagram
\[
\begin{tikzcd}
	\Q\rar \drar[swap]{id} & \eqclass\dar{t}\\
	& \Q
\end{tikzcd}
\]
The target map is a trivial fibration, hence by 2-out-of-3 the diagonal is a Joyal equivalence.
\end{proof}

\begin{sub}
We relate univalence to the Grothendieck construction.
\end{sub}

\begin{proposition}
The map
\[
\Fun(A,h(\Delta^1,\Q))\to \Fun(A,\eqclass)
\]
is a Joyal equivalence.
\end{proposition}

\begin{proof}
The target map
\[
h(\Delta^1,\Q)\to \Q
\]
is a trivial fibration by \cite[Corollary 3.5.10]{cisinskibook} and the target map
\[
\eqclass\to \Q
\]
is a trivial fibration by Proposition \ref{targetfibration}. Therefore, since trivial fibrations are stable under exponentiation and by 2-out-of-3, the map
\[
\Fun(A,h(\Delta^1,\Q))\to \Fun(A,\eqclass)
\]
is a Joyal equivalence.
\end{proof}

\begin{corollary}
For any pair of coCartesian fibrations
\[
\begin{tikzcd}
X\drar & & Y\dlar\\
& A & 
\end{tikzcd}
\]
classified by functors \(F,G\colon A\to Q\) respectively, the induced map
\[
k(A,\Q)(F,G)\to \equivfib_A(X,Y)
\]
is an equivalence of \(\infty\)-groupoids.
\end{corollary}

\begin{proof}
The map is obtained by taking the fiber of
\[
\begin{tikzcd}
	\Fun(A,h(\Delta^1,\Q))\ar{rr}\drar & & \Fun(A,\eqclass)\dlar\\
	& \Fun(A,\Q)\times \Fun(A,\Q) &
\end{tikzcd}
\]
which is a Joyal equivalence between Joyal fibration with target an \(\infty\)-category. Therefore the induced map on fibers is a Joyal equivalence.
\end{proof}

\begin{corollary}\label{grothendieckgroupoid}
The Grothendieck construction induces an equivalence on maximal groupoids.
\end{corollary}

\begin{proof}
It is clear that the Grothendieck construction is essentially surjective. Since
\[
k(A,\Q)(F,G)\to \equivfib_A(X,Y)
\]
is an equivalence of \(\infty\)-groupoids, taking \(\pi_0\) implies that it is also fully faithful.
\end{proof}


\section{Directed univalence}

\begin{sub}
Univalence relates internally defined equivalences of functors \(A\to \Q=\Q_\kappa\) with externally defined equivalence of coCartesian fibrations over \(A\). A directed version of univalence then should relate internally defined morphisms in \(\Fun(A,\Q)\) with externally defined morphisms of coCartesian fibrations over \(A\). The key ingredient here is the Grothendieck construction.
\end{sub}

\begin{proposition}\label{grothendieckcat}
The Grothendieck construction
\[
ho(\Fun(A,\Q))\to \coCart_\kappa(A)
\]
is an equivalence of categories.
\end{proposition}

\begin{proof}
We already know that it is essentially surjective, so we need to show that it is fully faithful. Let us first show that it is full. Let 
\[
\begin{tikzcd}
X^\natural\ar{rr}{f}\drar & & Y^\natural\dlar\\
& A^\sharp &
\end{tikzcd}
\]
be a morphism of coCartesian fibrations. By the Quillen equivalence mentioned in paragraph \ref{grothendieckquillen} there exists a coCartesian fibration
\[
\begin{tikzcd}
	W^\natural\dar{p}\\
	\mcyl \times A^\sharp
\end{tikzcd}
\]
and a natural equivalence
\[
\begin{tikzcd}
	W^\natural_0 \rar{u} \dar[swap]{\varphi} & X^\natural\dar{f}\\
	W^\natural_1 \rar{v} & Y^\natural
\end{tikzcd}
\]
in \(\mSet/A^\sharp\), where \(u\) and \(v\) are coCartesian equivalences. We already know that the Grothendieck construction induces an equivalence of maximal groupoids, Corollary \ref{grothendieckgroupoid}, so \(u\) and \(v\) are in the image of the Grothendieck construction. Therefore \(f\) is the image of \(v\varphi u^{-1}\).

It remains to show faithful. Note that a homotopy of maps in \(\mSet/A^\sharp\) between maps of coCartesian fibrations can be represented by a triangle
\[
\begin{tikzcd}
	X^\natural \times\mcyl \ar{rr}\drar[swap]{(p,1)} & & Y^\natural\times \mcyl\dlar{(q,1)}\\
	& A^\sharp \times \mcyl &
\end{tikzcd}
\]
Under the Quillen equivalence of paragraph \ref{grothendieckquillen} there is a coCartesian fibration 
\[
\begin{tikzcd}
	W^\natural\dar\\
	\mcyl \times\mcyl\times A^\sharp
\end{tikzcd}
\]
which is constant in the first variable and induces the homotopic maps in the second variable.
\end{proof}

\begin{corollary}\label{mappingspaceequivalence}
For any pair of coCartesian fibrations
\[
\begin{tikzcd}
X\drar & & Y\dlar\\
& A & 
\end{tikzcd}
\]
classified by functors \(F,G\colon A\to Q\) respectively, the induced map
\[
\Fun(A,\Q)(F,G)\to \Map^\sharp_A(X,Y)
\]
is an equivalence of \(\infty\)-groupoids.
\end{corollary}

\begin{proof}
Let \(K\) be a simplicial set and denote 
\[
F_K:= K\times A\to K \xrightarrow{F} \Q
\]
We then have
\[
\Fun(K,\Fun(A,\Q)(F,G))\simeq \Fun(K\times A,\Q)(F_K,G_K)
\]
On the other hand we have
\[
\Fun(K,\Map_A^\sharp(X,Y))\simeq \Map_{K\times A}^\sharp(K\times X, K\times Y)
\]
Thus Proposition \ref{grothendieckcat} provides an isomorphism
\[
\pi_0\Fun(K,\Map_A^\sharp(X,Y))\simeq \pi_0\Map_{K\times A}^\sharp(K\times X, K\times Y)
\]
functorial in all \(K\) which implies the assertion.
\end{proof}

\begin{corollary}\label{cocompleteq}
The \(\infty\)-category \(\Q\) is (\(\kappa\)-small) cocomplete.
\end{corollary}

\begin{proof}
Let \(I\) be a \(\kappa\)-small \(\infty\)-category. The canonical map \(\pi\colon I\to \Delta^0\) induces an adjunction
\[
\mathbf L\pi_!: \coCart_\kappa(I\times A) \rightleftarrows \coCart_\kappa(A) : \mathbf R\pi^*
\]
By Proposition \ref{grothendieckcat}, this induces an adjunction
\[
ho (\Fun(A, \Fun(I,\Q))) \rightleftarrows ho (\Fun(A,\Q))
\]
functorial in \(A\). This implies by \cite[Theorem 6.1.23]{cisinskibook} that the constant functor
\[
\Q\to \Fun(I,\Q)
\]
has a left adjoint, hence \(\Q\) has colimits of shape \(I\).
\end{proof}

\begin{remark}
The previous Corollary shows that the colimit of a functor \(I\to \Q\) is computed by inverting the coCartesian edges in the total space of the corresponding coCartesian fibration.
\end{remark}

\begin{theorem}[directed univalence]\label{diruniv}
The map \eqref{comparison}
\[
\begin{tikzcd}
	\Fun(\Delta^1,\Q)\ar{rr}\drar & & \morphclass\dlar\\
	& \Q \times \Q &
\end{tikzcd}
\]
is an equivalence of \(\infty\)-categories over \(\Q \times \Q\).
\end{theorem}

\begin{proof}
By \cite[Corollary 2.4.7.11]{lurie} the map 
\[
\Fun(\Delta^1,\Q)\to \Q\times \Q
\]
is a bifibration and by Corollary \ref{bifibration} the map
\[
\morphclass \to \Q\times \Q
\]
is a bifibration. By \cite[Proposition 2.4.7.7]{lurie} it is therefore enough to show that we have a fiberwise equivalence. The map on fibers is precisely the map of Corollary \ref{mappingspaceequivalence} with \(A=\Delta^0\).
\end{proof}

\begin{remark}\label{dirunivleft}
The same proof also shows that the universal left fibration is directed univalent.
\end{remark}

\begin{sub}
Suppose \(p\colon E\to B\) is a coCartesian fibration with small fibers. We have a pullback square
\[
\begin{tikzcd}
\morphclassi_{B\times B}(E^0,E^1)\rar \dar & \morphclass\dar\\
B\times B \rar & \Q \times \Q
\end{tikzcd}
\]
hence the comparison map
\[
\begin{tikzcd}
{} & \morphclass\dar\\
\Fun(\Delta^1,\Q)\rar \urar[dashed] & \Q\times \Q
\end{tikzcd}
\]
induces a map
\[
\begin{tikzcd}[column sep=small]
\Fun(\Delta^1,B)\drar[dashed]\ar{rr}\ar[bend right]{ddr} & & \Fun(\Delta^1,\Q)\drar\ar[bend right]{ddr} & \\
 & \morphclassi_{B\times B}(E^0,E^1)\dar \ar{rr} & & \morphclass\dar\\
 & B\times B \ar{rr} & & \Q\times \Q
\end{tikzcd}
\]
Thus, for any coCartesian fibration with small fibers a comparison map
\[
\begin{tikzcd}
{} & \morphclassi_{B\times B}(E^0,E^1)\dar\\
\Fun(\Delta^1,B)\rar \urar[dashed] & B\times B
\end{tikzcd}
\]
\end{sub}

\begin{definition}\label{def:directed univalent}
A coCartesian fibration \(p\colon E\to B\) with small fibers is called \emph{directed univalent} if the above comparison map is an equivalence of \(\infty\)-categories
\[\Fun(\Delta^1,B)\cong\morphclassi_{B\times B}(E^0,E^1)\, .\]
\end{definition}

\begin{theorem}\label{thm:directedunivfull}
Let \(p\colon E\to B\) be a coCartesian fibration of \(\infty\)-categories with small fibers. Then \(p\) is directed univalent if and only if its classifying
functor $B\to\Q$ is fully faithful.
\end{theorem}

\begin{proof}
Since it is a bifibration,
the map \(\Fun(\Delta^1,B)\rightarrow\morphclassi_{B\times B}(E^0,E^1)\)
is an equivalence of \(\infty\)-categories over \(B\times B\) if and only if it is a fiberwise equivalence. Consider the commutative diagram
\[
\begin{tikzcd}[column sep=small]
\Fun(\Delta^1,B)\drar\ar{rr}\ar[bend right]{ddr} & & \Fun(\Delta^1,\Q)\drar\ar[bend right]{ddr} & \\
 & \morphclassi_{B\times B}(E^0,E^1)\dar \ar{rr} & & \morphclass\dar\\
 & B\times B \ar{rr} & & \Q\times \Q
\end{tikzcd}
\]
Recall that the front square is a pullback square. Taking fibers at a point \((a,b)\colon \Delta^0\to B\times B\) yields the commutative triangle
\[
\begin{tikzcd}
B(a,b) \rar \drar & \Q(a^*E,b^*E)\dar\\
 & \Map^\sharp(a^*E,b^*E)
\end{tikzcd}
\]
Since the universal coCartesian fibration is directed univalent, the right hand vertical map is an equivalence. Thus by 2-out-of-3 \(p\) is directed univalent if and only if its classifying map is fully faithful.
\end{proof}

\begin{sub}
In Remark \ref{dirunivleft} we have already mentioned that the universal left fibration is directed univalent. As a consequence we obtain the following Corollary.
\end{sub}

\begin{corollary}
The classifying map of the universal left fibration \(\spaces\to \Q\) is fully faithful.
\end{corollary}


\section{Straightening/unstraightening as a consequence of directed univalence}

\begin{sub}
Here, we assume that \(\kappa\) is inaccessible.
The goal of this section is to prove that for any simplicial set \(A\), there is an equivalence of \(\infty\)-categories between the functor category \(\Fun(A,\Q_\kappa)\) and the \(\infty\)-categorical localization of the category \(\mSet_\kappa/A^\sharp\) of \(\kappa\)-small simplicial sets over \(A^\sharp\) at the weak
equivalence from Theorem \ref{cocartmodelstructure}.
In particular this implies that the \(\infty\)-category \(\Q_\kappa\) is equivalent to the \(\infty\)-categorical localization of simplicial sets at the Joyal equivalences.
\end{sub}

\begin{sub}
This section makes essential use of the theory of localizations of \(\infty\)-categories as developed in \cite[Chapter 7]{cisinskibook}. Given an \(\infty\)-category \(C\) together with a subcategory of weak equivalences, we will denote by \(L(C)\) the \(\infty\)-categorical localization of \(C\) at the weak equivalences. If \(C\) is an ordinary category with weak equivalences, the 1-categorical localization is obtained as \(C\to ho(L(C))\). A key element will be the following Theorem (we will also write
\(J\) for the nerve of any category \(J\)).
\end{sub}

\begin{theorem}\label{functorlocalization}
Let \(C\) be a cocomplete model category and \(I\) be a small category. Then the
induced comparison functor \(\Fun(I,C)\to\Fun(I,L(C))\)
defines an equivalence of \(\infty\)-categories
\[
L(\Fun(I,C))\xrightarrow{\simeq}\Fun(I,L(C))
\]
where the left hand side functor category is localized at the levelwise equivalences.
\end{theorem}

\begin{proof}
This is a direct consequence of \cite[Theorem 7.9.8 and Remark 7.9.7]{cisinskibook}.
\end{proof}

\begin{sub} Let \(\lambda\geq \kappa\) be any other inaccessible cardinal.
The theorem above implies that for any \(\lambda\)-small
category \(I\) and any \(\lambda\)-small simplicial set \(A\), the canonical functor
\begin{equation}\label{hocatlocaliztion}
\Fun(I,\mSet_\lambda/A^\sharp)\to ho\Fun(I,L(\mSet_\lambda/A^\sharp))
\end{equation}
is the (1-categorical) localization
of the functor category at the level wise coCartesian equivalences.
\end{sub}

\begin{sub}
We give another construction of this localization. Recall from \cite[Theorem 5.2]{cocartesiankim} that we have for any small category \(I\) and any simplicial set \(A\), a Quillen equivalence
\[
\Fun(I,\mSet/A^\sharp)\xrightarrow{\simeq} \mSet/I^\sharp \times A^\sharp
\]
Taking the homotopy category of the target and restricting to \(\lambda\)-small simplicial sets over \(A^\sharp\), we obtain
\[
\Fun(I,\mSet_\lambda/A^\sharp)\to \coCart_\lambda(I \times A)
\]
as the (\(1\)-categorical) localization of the functor category at the level wise coCartesian equivalences. Now recall from Proposition \ref{grothendieckcat} that the Grothendieck construction on homotopy categories provides an equivalence of categories
\[
\coCart_\lambda(I \times A)\simeq ho\Fun(I,\Fun(A,\Q_\lambda))
\]
To summarize, the composition
\begin{equation}\label{1localization}
\Fun(I,\mSet_\lambda/A^\sharp) \to \coCart_\lambda(I \times A)
\simeq ho\Fun(I,\Fun(A,\Q_\lambda))
\end{equation}
also defines the (1-categorical) localization of the functor category at the level wise coCartesian equivalences. Thus we have an equivalence of categories
\[
ho\Fun(I,L(\mSet_\lambda/A^\sharp)) \simeq ho\Fun(I,\Fun(A,\Q_\lambda))
\]
Since \(\Q_\kappa\) is a full subcategory of \(\Q_\lambda\), 
the functor category
\(\Fun(I,\Fun(A,\Q_\kappa))\) is the full subcategory of
\(\Fun(I,\Fun(A,\Q_\lambda))\) spanned by functors that take values in
\(\Q_\kappa\). This implies that the composition
\[
ho\Fun(I,L(\mSet_\kappa/A^\sharp))\to
ho\Fun(I,L(\mSet_\lambda/A^\sharp)) \simeq ho\Fun(I,\Fun(A,\Q_\lambda))
\]
factors through \(ho\Fun(I,\Fun(A,\Q_\kappa)\), thus induces
a functor
\begin{equation}\label{1localization2}
\Fun(I,\mSet_\kappa/A^\sharp)\to
ho\Fun(I,\Fun(A,\Q_\kappa))
\end{equation}
\end{sub}

\begin{sub}
Our next goal is to construct a comparison functor
\[
L(\mSet_\kappa/A^\sharp)\to \Fun(A,\Q_\kappa)
\]
To this end we consider \(I=\mSet_\kappa/A^\sharp\) and \(\lambda>\kappa\) so that \(I\)
is \(\lambda\)-small.
The image of the identity under the map \eqref{1localization2}
provides a functor
\[
\gamma_A\colon \mSet_\kappa/A^\sharp \to \Fun(A,\Q_\kappa)
\]
\end{sub}

\begin{lemma}
The functor \(\gamma_A\) sends coCartesian equivalences to equivalences of the functor category.
\end{lemma}

\begin{proof}
The functoriality of \eqref{1localization} in \(I\) implies that this localization functor is induced by post composition with \(\gamma_A\). In particular, this implies that \(\gamma_A\) sends coCartesian equivalences to equivalences of the functor category.
\end{proof}

\begin{sub}
Let \(A\) be a \(\kappa\)-small simplicial set.
We get an induced functor
\[
L\gamma_A\colon L(\mSet_\kappa/A^\sharp)\to \Fun(A,\Q_\kappa)
\]
so that, for any \(\kappa\)-small category \(I\),
\[
\begin{tikzcd}
\Fun(I,\mSet_\kappa/A^\sharp)\rar{(\gamma_A)_\ast}\dar & \Fun(I,\Fun(A,\Q_\kappa))\\
\Fun(I,L(\mSet_\kappa/A^\sharp))\urar[swap]{(L\gamma_A)_\ast} &
\end{tikzcd}
\]
commutes. Taking homotopy categories, the vertical arrow is the localization \eqref{hocatlocaliztion}, while the horizontal arrow is the localization \eqref{1localization} for \(\lambda=\kappa\). Thus we have shown the following Lemma.
\end{sub}

\begin{lemma}\label{hocateq}
The functor
\[
(L\gamma_A)_\ast \colon \Fun(I,L(\mSet_\kappa/A^\sharp)) \to \Fun(I,\Fun(A,\Q_\kappa))
\]
induces an equivalence on homotopy categories.
\end{lemma}

\begin{sub}
In particular we have that
\[
L\gamma_A\colon L(\mSet_\kappa/A^\sharp) \to \Fun(A,\Q_\kappa)
\]
induces an equivalence of homotopy categories. Our next goal is to show that it is actually an equivalence of \(\infty\)-categories. To this end, first note
that \(L(\mSet_\kappa/A^\sharp)\) is a cocomplete \(\infty\)-category since
it is the localization of a model category,
see for example \cite[Theorem 2.5.9]{bhhprocats}
or \cite[Proposition 7.7.4]{cisinskibook}. On the other hand, the functor category \(\Fun(A,\Q_\kappa)\) is cocomplete by Corollary \ref{cocompleteq}. 
\end{sub}

\begin{lemma}\label{cocontinuous}
The functor \(L\gamma_A\) preserves \(\kappa\)-small colimits.
\end{lemma}

\begin{proof}
First note that by \cite[Proposition 7.3.23]{cisinskibook},
it suffices to show that \(L\gamma_A\) preserves colimits indexed by (nerves of)
\(\kappa\)-small categories.
Thus, let \(I\) be a \(\kappa\)-small category. We have the commutative square
\begin{equation}\label{bcsquare}
\begin{tikzcd}[column sep=large]
L(\mSet_\kappa/A^\sharp)\rar{L\gamma_A} \dar & \Fun(A,\Q_\kappa)\dar\\
\Fun(I,L(\mSet_\kappa/A^\sharp)) \rar{(L\gamma_A)_\ast} & \Fun(I,\Fun(A,\Q_\kappa))
\end{tikzcd}
\end{equation}
where the vertical arrows are given by taking constant diagrams.
Since \(L(\mSet_\kappa/A^\sharp)\) and \(\Fun(A,\Q_\kappa)\) are cocomplete, the vertical arrows admit left adjoints. The functor \(L\gamma_A\) preserves colimits if and only if the induced square after taking left adjoints commutes (up to equivalence). In other words, if and only if the square (\ref{bcsquare}) satisfies the Beck-Chevalley condition. This can be checked by taking homotopy categories. By Lemma \ref{hocateq}, the horizontal arrows become equivalences after taking homotopy categories, so that the induced square does indeed satisfy the Beck-Chevalley condition, hence \(L\gamma_A\) preserves colimits.
\end{proof}

\begin{sub}
This already suffices to prove the main result of this section.
\end{sub}

\begin{theorem}\label{thm:straightening}
The map \(L\gamma_A\) induces an equivalence of \(\infty\)-categories
\[
L(\mSet_\kappa/A^\sharp) \cong \Fun(A,\Q_\kappa)\, .
\]
\end{theorem}

\begin{proof}
By Lemma \ref{hocateq} and Lemma \ref{cocontinuous}, \(L\gamma_A\) is a colimit preserving functor between cocomplete \(\infty\)-categories
inducing an equivalence of homotopy categories.
Therefore, by virtue of \cite[Corollary 3.3.5]{nrsadjoint} or \cite[Theorem 7.6.10]{cisinskibook} it is an equivalence of \(\infty\)-categories.
\end{proof}

\begin{remark}
By virtue of \cite[Proposition 7.10.13]{cisinskibook}, the theorem above allows
to compute the mapping space of the \(\infty\)-category \(\Fun(A,\Q)\) in terms
of the mapping space of the coCartesian model structure on \(\mSet_\kappa/A^\sharp\):
given two coCartesian fibrations \(X\to A\) and \(Y\to A\) classified by two functors
\(F:A\to\Q\) and \(G:A\to\Q\), respectively, the mapping space
\(\Fun(A,\Q)(F,G)\) is canonically equivalent to the maximal Kan complex of
the \(\infty\)-category \(\Map^\sharp_A(X,Y)\)
of coCartesian functors \(X\to Y\) over \(A\).
From there, we deduce
easily that the
identification of Corollary \ref{mappingspaceequivalence}
is indeed induced by the equivalence of Theorem \ref{thm:straightening}.
In particular, it is functorial in each variable.
Since the homotopy coherent nerve of the simplicial category of fibrant
objects associated to the model structure \(\mSet_\kappa/A^\sharp\)
is an explicit model of the localization \(L(\mSet_\kappa/A^\sharp)\)
by \cite[Prop.~1.2.1]{hinich},
this also means that the \(\infty\)-category of \(\kappa\)-small
\(\infty\)-categories considered by Lurie \cite{lurie,kerodon} is indeed
(equivalent to) \(\Q_\kappa\).
\end{remark}

\begin{sub}
Let \(\mathcal{A}\) be the object of \(\Q=\Q_\kappa\)
corresponding to \(A\), so that we have
a pullback square of the form below.
\[
\begin{tikzcd}
A\ar{r}\ar{d}&\Q_{\bullet}\ar{d}{\quniv}\\
\Delta^0\ar{r}{\mathcal{A}}&\Q
\end{tikzcd}
\]
By virtue of \cite[Corollary 7.6.13]{cisinskibook},
the equivalence \(L(\mSet_\kappa)\simeq\Q\) induces an equivalence of
\(\infty\)-categories
\begin{equation}\label{eq:slicecat}
L(\mSet_{\kappa/A})\cong \Q_{/\mathcal{A}}
\end{equation}
where \(\mSet_{\kappa/A}\) is the category of \(\kappa\)-small
marked simplicial
sets over \(A^\sharp\) equipped with the sliced model structure
(where the weak equivalences are defined as those maps
that are weak equivalences in \(\mSet\)).
Therefore, up to isomorphism,
any object of \(\Q_{/\mathcal{A}}\), seen as a map
\(p\colon \mathcal{X}\to \mathcal{A}\),
comes from an isofibration \(\tilde p\colon X\to A\) in \(\mSet_\kappa\).
The property that \(\tilde p\) is a coCartesian fibration is well defined
up to equivalence; this is thus a property of \(p\).
Similarly any map in \(\Q_{/\mathcal{A}}\) can be described as a map
between isofibrations over \(A\), and the property that
such a map is coCartesian is a well defined property
of the corresponding map in \(\Q_{/\mathcal{A}}\).
We denote by \(\coCarti(\mathcal{A})=\coCarti_\kappa(\mathcal{A})\)
the $\infty$-subcategory of slice \(\infty\)-category \(\Q_{/\mathcal{A}}\)
who objects are coCartesian fibrations and maps are coCartesian;
in other words, as a simplicial set, the \(n\)-simplices of \(\coCarti(A)\)
are those functors \(u\colon\Delta^n\to\Q_{/\mathcal{A}}\) such that, for \(0<i\leq n\),
the restriction of \(u\) to \(\Delta^{\{i-1,i\}}\) corresponds to a commutative
triangle of \(\Q\) of the form
\[
\begin{tikzcd}
\mathcal{X}\ar{rr}{f}\drar[swap]{p} & & \mathcal{Y}\dlar{q}\\
& \mathcal{A} &
\end{tikzcd}
\]
with \(p\) and \(q\) coCartesian fibrations, and \(f\) a
functor that preserves coCartesian edges. 
\end{sub}

\begin{theorem}\label{thm:straightening2}
For any \(\kappa\)-small \(\infty\)-category \(A\),
corresponding to an object \(\mathcal{A}\) of \(\Q\),
there is a canonical
equivalence of \(\infty\)-categories
\[
\coCarti(\mathcal{A})\cong\Fun(A,\Q)\, .
\]
\end{theorem}

\begin{proof}
Let us write \(\coCarti^\prime(A)\) for the full
subcategory of \(\mSet_\kappa/A^\sharp\) spanned by coCartesian
fibrations \(X\to A\) (with \(\kappa\)-small fibers).
Inverting fiberwise equivalences induces an equivalence of \(\infty\)-categories
\[
L(\coCarti^\prime(A))\cong L(\mSet_\kappa/A^\sharp)
\]
and thus an equivalence of \(\infty\)-categories
\[
L(\coCarti^\prime(A))\cong \Fun(A,\Q)
\]
by Theorem \ref{thm:straightening}.
By virtue of \cite[Theorem 7.5.18]{cisinskibook}, the localization functor
(from equivalence \eqref{eq:slicecat} above)
\[
\mSet_{\kappa/A}\to\Q_{/\mathcal{A}}
\]
is left exact in the sense of \cite[Definition 7.5.2]{cisinskibook}.
Therefore, since the embedding functor
\(\coCarti(\mathcal{A})\to\Q_{/\mathcal{A}}\)
is conservative and preserves finite limits (it obviously preserves
the terminal object as well a pullbacks, so that we may apply
\cite[Proposition 7.3.28]{cisinskibook}, for instance),
the preceding functor induces a functor
\[
\coCarti^\prime(A)\to\coCarti(\mathcal{A})
\]
which is left exact as well (observe that the
model structure of Theorem \ref{cocartmodelstructure}
is a left Bousfield localization of the sliced model
structure). The homotopy category
\[
ho(L(\coCarti^\prime(A)))\cong\coCart_\kappa(A)
\]
is easy to describe:
the objects are coCartesian fibrations with \(\kappa\)-small fibers
\(X\to A\) and maps are homotopy classes of maps in \(\mSet\)
over \(A^\sharp\).
It is an elementary observation that the induced
functor 
\[
ho(L(\coCarti^\prime(A)))\to ho(\coCarti(\mathcal{A}))
\]
is an equivalence of categories.
Therefore, by virtue of \cite[Corollary 3.3.5]{nrsadjoint}
or \cite[Theorem 7.6.10]{cisinskibook},
the induced functor
\[
\Fun(A,\Q)\cong L(\coCarti^\prime(A))\to\coCarti(\mathcal{A})
\]
is an equivalence of \(\infty\)-categories.
\end{proof}

\begin{remark}
The title of this section claims straightening as a consequence
of directed univalence, but we never explicitly make the connection.
Instead we use the Grothendieck construction, i.e. the fact that
we have a functorial equivalence on homotopy categories between
the cocartesian model structure and the functor category (Proposition \ref{grothendieckcat}). In the previous section,
we used this to prove directed univalence (Theorem \ref{diruniv}),
but it is an easy exercise to
see that Proposition \ref{grothendieckcat} follows straight away from Theorem \ref{diruniv}.
In other words, directed univalence is equivalent to Proposition \ref{grothendieckcat}. Since
Theorem \ref{diruniv} is the poor man version of the Straightening/Unstraightening correspondence
(formulated as Theorem \ref{thm:straightening} or Theorem \ref{thm:straightening2}), all this means that
directed univalence expresses nothing else than Straightening/Unstraightening.
\end{remark}

\bibliographystyle{alpha}
\bibliography{directreferences}

\end{document}